\documentclass[preprint,12pt,3p]{elsarticle}

\usepackage{graphicx}
\usepackage{subcaption}
\usepackage{amssymb, amsmath}
\usepackage{amsthm}
\usepackage{color}
\usepackage{array}
\usepackage{tabu}

\journal{Regular \& Chaotic Dynamics}

\newtheorem{theorem}{Theorem}[section]

\theoremstyle{definition}

\theoremstyle{remark}

\begin{document}

\begin{frontmatter}


\title{The method of averaging for the Kapitza-Whitney pendulum}



\author[label5]{Ivan Polekhin}
\address[label5]{Steklov Mathematical Institute of the Russian Academy of Sciences}


\begin{abstract}

A generalization of the classical Kapitza pendulum is considered: an inverted planar mathematical pendulum with a vertically vibrating pivot point in a time-periodic horizontal force field. We study the existence of forced oscillations in the system. It is shown that there always exists a periodic solution along which the rod of the pendulum never becomes horizontal, i.e. the pendulum never falls, provided the period of vibration and the period of horizontal force are commensurable. We also present a sufficient condition for the existence of at least two different periodic solutions without falling. We show numerically that there exist stable periodic solutions without falling.

\end{abstract}


\end{frontmatter}


\section{Introduction}

A planar inverted mathematical pendulum with a vibrating pivot point is a classical mechanical system. This dynamical system has been studied thoroughly by many authors starting from the works of A. Stephenson \cite{stephenson1908xx}, P.\,L.\,Kapitza \cite{kapitsa1951pendulum,kapitsa1951dynamic}, and N.\,N.\,Bogolyubov \cite{bogolyubov1950perturbation} on the stabilization of the vertical equilibrium, and ending with the work of D. Acheson \cite{acheson1995multiple}, where the existence of the so-called nodding solutions has been shown numerically. A more detailed overview of papers related to the subject, including the history of the problem and the role which this system played in nonlinear studies and applications, can be found in \cite{burd2007method,samoilenko1994nn,butikov2001dynamic,wright2017comparisons}.

A less known system, also related to pendulum dynamics, is an inverted pendulum with a horizontally moving pivot point. The law of motion of the base is assumed to be a given function of time. H. Whitney was the first who proposed that, for any given law of motion of the pivot, there exists an initial position of the pendulum such that the rod of the pendulum, being released from this position with zero generalized velocity, always remains above the horizontal line during the motion of the system (i.e. never falls) \cite{courant1996mathematics}. Rigorous proofs of this statement had been obtained later (see, for instance, \cite{broman1958mechanical,polekhin2014examples}). The history of this problem can be found in detail in \cite{srzednicki2019periodic}.

In the paper we consider a generalization of both the Kapitza and the Whitney pendulum: a system of an inverted pendulum with a rapidly vertically oscillating pivot point in a horizontal external non-autonomous force field. For the Whitney pendulum this horizontal force is the force of inertia corresponding to the horizontal motion of the base. The dynamics of the Kapitza-Whitney pendulum obeys the following equation

\begin{align}
    \label{eq1}
    \ddot x + \mu \dot x + (1 + \ddot f(t)) \sin x + h(t) \cos x = 0.
\end{align}
Here $x$ stands for the inclination of the rod in such a way that the position $x = \pi$ corresponds to the vertical upward configuration of the pendulum. The units of measurement are chosen so that the mass and the length of the pendulum and the gravity acceleration equal $1$; $\mu \geqslant 0$ is the viscous friction coefficient; functions $f(t)$ and $h(t)$ defines the vertical position of the pivot point and the horizontal external force, correspondingly. A typical example of function $f(t)$, that allows one to refer to the classical results on averaging, is the following law of motion
\begin{align}
    \label{eq_ft}
f(t) = \frac{1}{k} \sin (\omega k t).
\end{align}
Here $\omega$ is a given parameter that defines the frequency of the vibration. Everywhere below we assume that $\omega$ and $k$ are natural numbers and $k$ is relatively large ($1/k$ is a small parameter for the averaging).

If we put $h \equiv 0$ in \eqref{eq1}, then we obtain the equation for the Kapitza pendulum. If, on the other hand, we put $f \equiv 0$, then our equation coincides with the one for the Whitney pendulum. We will be interested in the case when $h \not\equiv 0$ and $\ddot f \not\equiv 0$. Moreover, we assume that these functions of time have common period, i.e. their frequencies are commensurable. For instance, in the simplest case, we can consider $2\pi$-periodic function $h(t)$. 
The main result of the paper is that for any sufficiently regular function $h(t)$ there exists a periodic solution of the system such that the rod never falls on the ground along this solution. Under some additional assumptions, we prove the multiplicity of such solutions: we show that there exist at least two different periodic solutions without falls. We also show numerically that, for some given $f(t)$ and $h(t)$, there exist stable periodic solutions without falls.

The main content of the paper is split into two sections. In the first part we prove the existence of a periodic solution without falls for the Kapitza-Whitney pendulum and present sufficient conditions for the existence of two such solutions. In the second part we numerically study the stability of periodic solutions without falls.

A classical theorem of N.N. Bogolyubov on averaging on a finite interval of time is the key ingredient of the proofs in \cite{bogolyubov1945statistical} (see also \cite{bogolyubov1961asymptotic}). We show how it is possible to move from the local results on averaging to the global ones. The correspondence between our results and the classical theorems on averaging is shortly outlined in the conclusion where we also present possible ways of generalization and development of our approach.

\section{Main results}
Let us consider the following system, a generalization of \eqref{eq1}:
\begin{align}
\begin{split}
\label{eq2}
    & \dot q = p - \varphi(kt) \sin q,\\
    & \dot p = -\mu p + [\mu \sin q + p \cos q] \varphi(kt) - \sin q - \frac{\varphi^2(kt)}{2} \sin 2q + h(q,p,t).
\end{split}
\end{align}
Here and below we assume that all functions are $C^\infty$-smooth. We also assume that $\varphi$ is a $T$-periodic function of its argument with zero average value (i.e. $\varphi(kt)$ has the period $T/k$), $h$ is a bounded and $T$-periodic in $t$ function. In order to obtain equation \eqref{eq1} with the vertical position term specified by the equation \eqref{eq_ft} from system \eqref{eq2}, one should put $\varphi(kt) = \omega \cos (\omega k t)$ and for $h(q, p, t)$ put $-h(t) \cos q$. Note that variable $q$ in \eqref{eq2} corresponds to $x$ in the original equation. The correspondence between equation \eqref{eq1} (Newton's law of motion) and system \eqref{eq2} (in a Hamiltonian form) is explained in detail in \cite{burd2007method}. The form of system \eqref{eq2} allows one to apply the classical theorem on averaging (for $h \equiv 0$) considering $1/k$ as the small parameter.

\begin{theorem}
Let $h$ satisfy the following inequalities for all $t$
\begin{align}
\label{eq3}
    h(\pi/2,0,t) < 1, \quad h(3\pi/2,0,t) > -1.
\end{align}
Then there exists $K$, such that for any natural number $k > K$ there is a $T$-periodic solution $(q(t), p(t))$ of \eqref{eq2} and $q(t) \in (\pi/2, 3\pi/2)$ for all $t$.
\end{theorem}
\begin{proof}

The main idea of the proof is to consider a modified system that differs from \eqref{eq2} on a compact subset of the extended phase space. Then, from the theory of upper and lower solutions for second order boundary value problems, one can show that the modified system has a periodic solution with the required properties. Finally, we show that, for $k$ sufficiently large, this periodic solution cannot go through the region of modification. Therefore, this solution exists in the original system. Now we turn to the details of the proof.

Let us consider the following modification of system \eqref{eq2}
\begin{align}
\begin{split}
\label{eq4}
    & \dot q = p - \sigma(q) \varphi(kt) \sin q,\\
    & \dot p = -\mu p + \sigma(q) [\mu \sin q + p \cos q]  \varphi(kt) - \sin q - \frac{\varphi^2(kt)}{2} \sin 2q + h(q,p,t).
\end{split}
\end{align}
Here $\sigma$ is a smooth function such that $|\sigma| \leqslant 1$. 
Suppose that this system has a $T$-periodic solution. First, let us show that for this solution there exists an \textit{a priori} estimation for $p(t)$, independent on $k$ and the form of function $\sigma$. We will use this estimation below for averaging.

Let $q(t)$ be a $T$-periodic solution. From the periodicity, we have that for some $t' \in [0,T]$ the derivative equals zero: $\dot q(t') = 0$. From the first equation of system \eqref{eq4} we obtain that  $|p(t')| \leqslant c_1$, where $c_1$ depend only on the maximum value of function $|\varphi|$. From the second equation of the system we have $\dot p \leqslant c_2 p + c_3$ and $\dot p \geqslant -c_2 p - c_3$, where non-negative constants $c_2$ and $c_3$ depends on $\mu$ and maximum values of $|\varphi|$ и $|h|$. Hence, the maximum of $|p(t)|$ is less than some constant $c$ that can be expressed by means of the quantities $c_1$, $c_2$, and $c_3$.

Introducing a small parameter $\varepsilon = 1/k$, let us rewrite system \eqref{eq4} in a standard form used for averaging
\begin{align}
\begin{split}
\label{eq5}
    & q' = \varepsilon (p - \sigma(q) \varphi(\tau) \sin q),\\
    & p' = \varepsilon(-\mu p + \sigma(q) [\mu \sin q + p \cos q]  \varphi(\tau) - \sin q - \frac{\varphi^2(\tau)}{2} \sin 2q + h(q,p,t)),\\
    & t' = \varepsilon.
\end{split}
\end{align}
Here $\tau = tk$, and $(\cdot)' = d/d\tau$.
The averaged system takes the form
\begin{align}
\begin{split}
\label{eq6}
    & q' = \varepsilon p,\\
    & p' = \varepsilon(-\mu p - \sin q - \frac{\Phi}{2} \sin 2q + h(q,p,t)),\\
    & t' = \varepsilon,
\end{split}
\end{align}
where $\Phi$ is the average value of function $\varphi^2(\tau)$ over its period. Let $\Delta > 0$ and $\delta > 0$ be such numbers that there exists $L > 0$ and for any initial condition $(q_0, p_0, t_0)$ satisfying 
$$
t_0 \in [0,T], \quad p_0 \in [-2c,2c], \quad q_0 \in [\pi/2,\pi/2 + \delta] \cup [3\pi/2 - \delta,3\pi/2]
$$
for some $l \in [0, L]$ for the corresponding solution of \eqref{eq6} one of the following conditions holds:
\begin{align*}
    &q(t_0 + l/\varepsilon) \leqslant \pi/2 - \Delta,\\
    &q(t_0 + l/\varepsilon) \geqslant 3\pi/2 + \Delta,\\
    &q(t_0 - l/\varepsilon) \leqslant \pi/2 - \Delta,\\
    &q(t_0 - l/\varepsilon) \geqslant 3\pi/2 + \Delta.
\end{align*}
In other words, any solution of the averaged system starting near the boundary of the set $q \in [\pi/2, 3\pi/2]$ leaves (in direct or reversed time) $\varepsilon$-neighborhood (w.r.t. $q$) of this subset of the extended phase space in time $\tau$ that is less or equal to $L$. Note that $c$ here is the constant that was obtained as an \textit{a priori} estimation for $|p(t)|$ of a $T$-periodic solution of the modified system.

The existence of such $\Delta$ and $\delta$ easily follows from the Taylor expansion for solutions of system \eqref{eq6} with the corresponding initial conditions $t_0 \in [0,T]$, $p_0 \in [-2c,2c]$ and $q_0 = \pi/2$ or $q_0 = 3\pi/2$. In particular, the condition \eqref{eq3} is used in this part of the proof.

Let us consider function $\sigma(q)$ of the following form:
$$
\sigma(q)=
\begin{cases}
& 0, \quad q \in [-\delta/2 + \pi/2, \pi/2 + \delta/2] \cup [-\delta/2 + 3\pi/2, 3\pi/2 + \delta/2],\\
& 1, \quad q \not\in [-\delta + \pi/2, \pi/2 + \delta] \cup [-\delta + 3\pi/2, 3\pi/2 + \delta],\\
&\mbox{monotonous elsewhere}.
\end{cases}
$$
Now we show that, for this $\sigma$ and for any given natural number $k$, system \eqref{eq4} admits a $T$-periodic solution. From \eqref{eq4} we have
\begin{align}
\begin{split}
\label{eq7}
    \ddot q = &-\frac{\partial \sigma}{\partial q}\dot q\varphi(kt)\sin q + \sigma(q) \frac{\partial \varphi(kt)}{\partial t} \sin q + \sigma(q) \varphi(kt) \cos q \dot q \\
    &-\mu [\dot q + \sigma(q) \varphi(kt) \sin q] + \sigma(q) [\mu \sin q + (\dot q + \sigma(q) \varphi(kt) \sin q)\cos q]\varphi(kt)\\
    &-\sin q - \frac{1}{2} \varphi^2(kt) \sin 2q + h(q,\dot q + \sigma(q) \varphi(kt) \sin q,t).
\end{split}
\end{align}
By a simple direct calculation one can check that $q = \pi/2$ and $q = 3\pi/2$ are lower and upper solutions (see, for instance, \cite{de2006two} or \cite{bernfeld1974introduction}) for our system \eqref{eq7}. Therefore, there exists a $T$-periodic solution and $q(t) \in (\pi/2,3\pi/2)$ for $t$.

Finally, let us show that, for sufficiently large $k$, this periodic solution cannot go through the points where $\sigma \ne 0$. Indeed, let $M$ be a compact of the form
$$
M = \{ q,p,t \colon -1 + \pi/2 \leqslant q \leqslant 1 + 3\pi/2, -2c \leqslant p \leqslant 2c, -1 \leqslant t \leqslant T + 1 \}.
$$
We can apply a classical theorem on averaging on a finite time interval for compacts \cite{sanders2007averaging}. From this theorem, we have that for large $k$ any solution of \eqref{eq7} that goes through a point where $\sigma \ne 0$ either leaves the interval $q \in (\pi/2, 3\pi/2)$ in time less or equal to $L$, or this solution was outside this interval earlier.

This contradicts the fact that $q(t) \in (\pi/2,3\pi/2)$. Therefore, we can conclude that the same periodic solution exists in the original system.

\end{proof}

When function $h$ satisfies some additional conditions, it is possible to prove that there are at least two periodic solutions of \eqref{eq2}, provided $k$ is large enough. Similar to the above, each of these solutions satisfies condition $q(t) \in (\pi/2,3\pi/2)$.

Indeed, condition \eqref{eq3} is the cornerstone of the proof of Theorem 2.1. This condition allows one to prove that there exists a periodic solution that always remains in interval $(\pi/2, 3\pi/2)$. To be more precise, we use that for function 
$$
f(q) = -\sin q - \frac{\Phi}{2} \sin 2q
$$
we have $f(\pi/2) = -1$ and $f(3\pi/2) = 1$, i.e. the values of the function have different signs at the ends of the interval. If $\Phi > 1$, then function $f$ has two local maxima inside the interval $(\pi/2, 3\pi/2)$ and $f > 0$ at these points; $f$ also has two local minima where $f < 0$.

Let us introduce the following notations
$$
\lambda_1 = \frac{-1 + \sqrt{1 + 8\Phi^2}}{4\Phi}, \quad \lambda_2 = \frac{-1 - \sqrt{1 + 8\Phi^2}}{4\Phi}.
$$
For $\Phi \in (0, 1)$ two critical points (inside $[0, 2\pi]$) of $f(q)$ are as follows
$$
q_{min}^1 = \mathrm{arccos}(\lambda_1), \quad q_{max}^1 = 2\pi - \mathrm{arccos}(\lambda_1).
$$
As $\Phi$ tends to $0$, value $q_{min}^1$ tends to $\pi/2$ and $q_{max}^1$ tends to $3\pi/2$. As $\Phi$ tends to $\infty$, $q_{min}^1$ tends to $\pi/4$ and $q_{max}^1$ tends to $7\pi/4$. If $\Phi > 1$, then we have two more additional critical points
$$
q_{max}^2 = \mathrm{arccos}(\lambda_2), \quad q_{min}^2 = 2\pi - \mathrm{arccos}(\lambda_2).
$$
As $\Phi$ tends to $1$, $q_{min}^2$ tends to $\pi$ and $q_{max}^2$ also tends to $\pi$. As $\Phi$ tends to $\infty$, $q_{min}^2$ tends to $5\pi/4$ and $q_{max}^2 $ tends $ 3\pi/4$. 

Similarly to Theorem 2.1, we can consider two intervals $[\pi/2, q_{max}^2]$ and $[q_{min}^2, 3\pi/2]$ independently and prove the following multiplicity result.

\begin{theorem}
Let $\Phi > 1$ and $h$ satisfy the following conditions for all $t$
\begin{align}
    h(t,\pi/2,0) < 1, \quad h(t, 3\pi/2,0) > -1, \quad h(t,q_{max}^2,0) > -f(q_{max}^2), \quad h(t,q_{min}^2,0) < -f(q_{min}^2).
\end{align}
Then there exists $K$ such that for any $k > K$, $k \in \mathbb{N}$ there are two $T$-periodic solutions $q_1(t)$ and $q_2(t)$ satisfying $q_1(t) \in (\pi/2, q_{max}^2)$ and $q_2(t) \in (q_{min}^2, 3\pi/2)$ for all $t$.
\end{theorem}

\section{Numerical results}

In this section we present some results concerning the study of stability of solutions without falls for the Kapitza-Whitney pendulum. To be more precise, we show numerically that stable and asymptotically stable periodic solutions without falls exist.

We will consider the following system
\begin{align*}
    & \dot q = p - \dot f \sin q,\\
    & \dot p = -\mu p + [\mu \sin q + p \cos q] \dot f - \sin q - \frac{\dot f^2}{2} \sin 2q + h(t) \cos q.
\end{align*}
Here $f(t) = \frac{1}{k}\sin(\omega k t)$ and assume that the external horizontal force has the form
\begin{align*}
    h(t) = c + A \sin (t).
\end{align*}
Below one can find asymptotically stable $2\pi$-periodic solutions for given $c$, $A$, $k$, $\omega$ and $\mu$ (Fig. 1 and Fig. 2). These solutions are limit cycles for the system. For each of these cycles we can conclude, based on the results of calculations, that the solution satisfying $q(0) = \pi$ and $p(0) = 0$ asymptotically tends to the corresponding one-dimensional invariant manifolds. We also present corresponding periodic solutions of the averaged system, i.e. the trajectories that correspond to the case $k = \infty$ (formally).
\begin{figure}[h!]
\begin{subfigure}{.5\textwidth}
  \centering
  \includegraphics[width=1.0\linewidth]{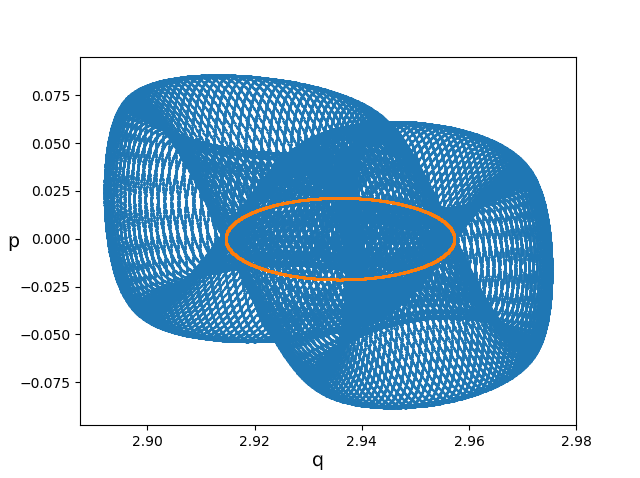}
  \caption{$k = 10$, $\omega = 10$, $A = 1$, $c=10$, $\mu = 1$.}
  \label{fig:sfig2}
\end{subfigure}%
\begin{subfigure}{.5\textwidth}
  \centering
  \includegraphics[width=1.0\linewidth]{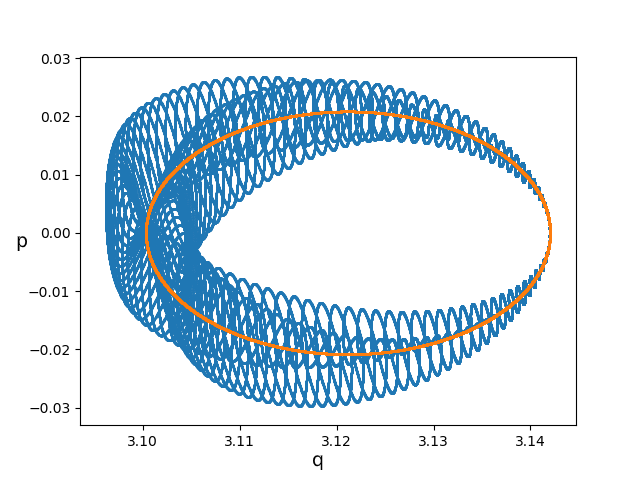}
  \caption{$k = 10$, $\omega = 10$, $A = 1$, $c=1$, $\mu = 1$.}
  \label{fig:sfig1}
\end{subfigure}
\caption{Asymptotically stable $2\pi$-periodic solutions and the corresponding solutions of the averaged system (highlighted). Along these solutions, the rod of the pendulum never becomes horizontal, i.e. $q(t) \in (\pi/2, 3\pi/2)$.}
\end{figure}

\begin{figure}[h!]
\begin{subfigure}{.5\textwidth}
  \centering
  \includegraphics[width=1.0\linewidth]{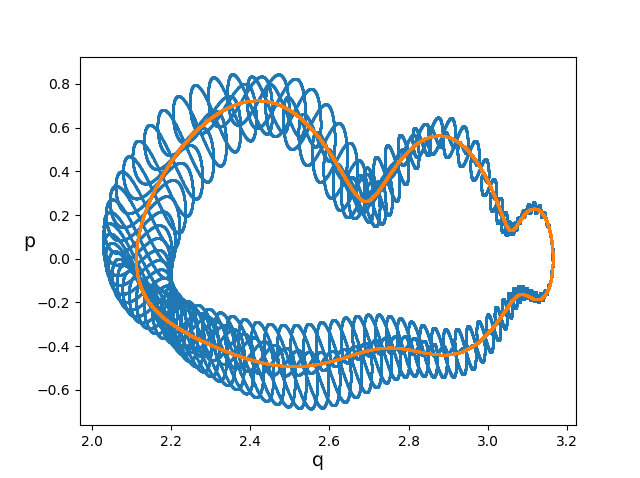}
  \caption{$k = 10$, $\omega = 10$, $A = 20$, $c=20$, $\mu = 1$.}
  \label{fig:sfig2}
\end{subfigure}%
\begin{subfigure}{.5\textwidth}
  \centering
  \includegraphics[width=1.0\linewidth]{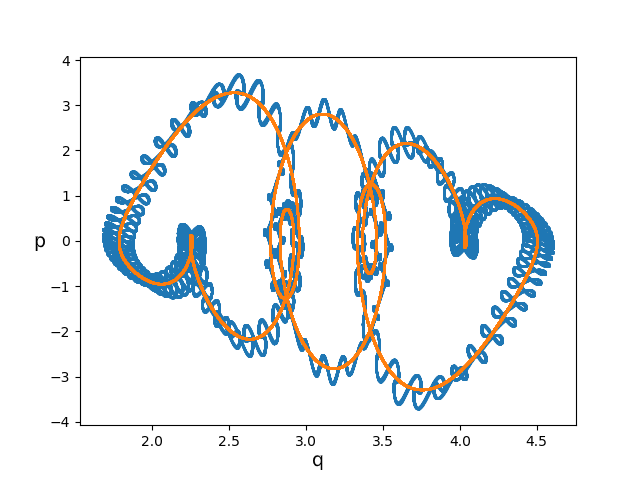}
  \caption{$k = 10$, $\omega = 15$, $A = 100$, $c=0$, $\mu = 1$.}
  \label{fig:sfig1}
\end{subfigure}\\
\begin{subfigure}{.5\textwidth}
  \centering
  \includegraphics[width=1.0\linewidth]{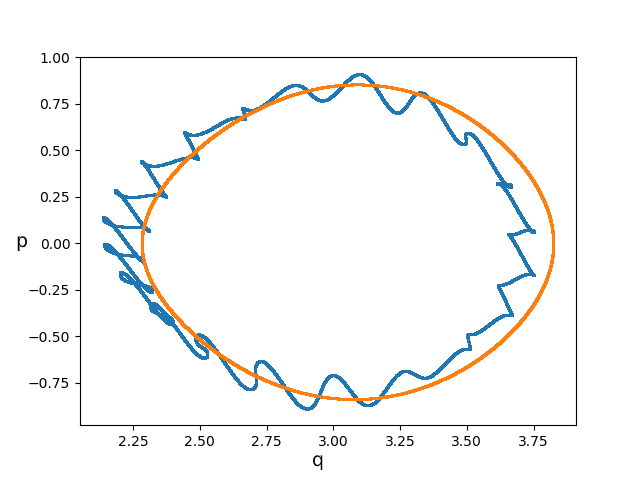}
  \caption{$k = 10$, $\omega = 2$, $A = 1$, $c=0$, $\mu = 1$.}
  \label{fig:sfig2}
\end{subfigure}%
\begin{subfigure}{.5\textwidth}
  \centering
  \includegraphics[width=1.0\linewidth]{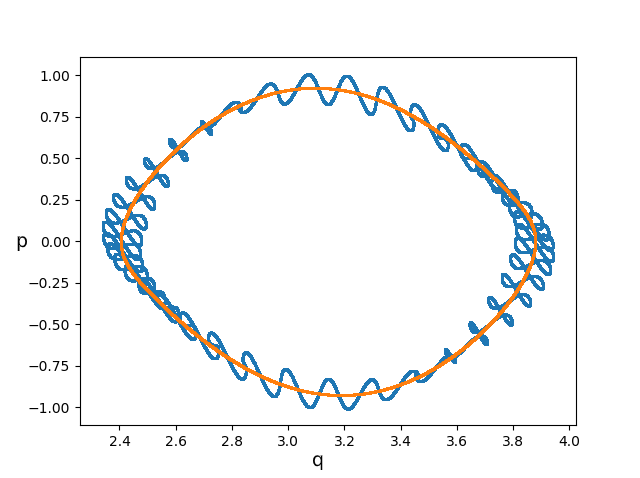}
  \caption{$k = 10$, $\omega = 4$, $A = 4$, $c=0$, $\mu = 1$.}
  \label{fig:sfig1}
\end{subfigure}\\
\label{figEarth}
\caption{Asymptotically stable $2\pi$-periodic solutions and the corresponding solutions of the averaged system (highlighted). Along these solutions, the rod of the pendulum never becomes horizontal, i.e. $q(t) \in (\pi/2, 3\pi/2)$.}
\end{figure}

When there is no friction in the system, solutions cannot be asymptotically stable. Nevertheless, it is still possible to study the stability based on the view of the Poincar{\'e} section in a neighborhood of a periodic solution.

\begin{figure}[h!]
\centering
\begin{subfigure}{.46\textwidth}
  \centering
  \includegraphics[width=1.0\linewidth]{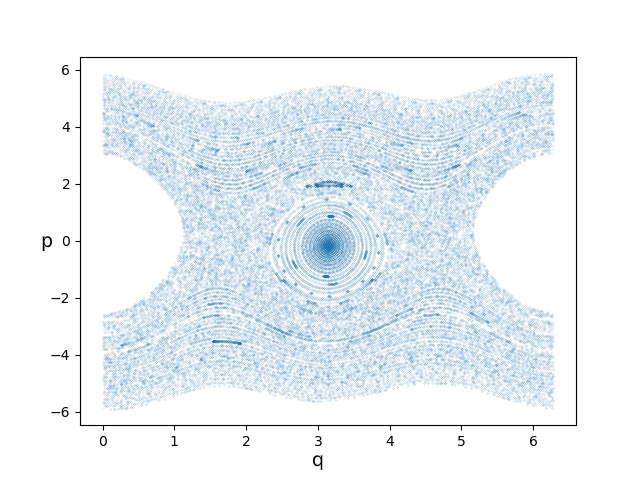}
  \caption{}
  \label{fig:sfig2}
\end{subfigure}%
\begin{subfigure}{.46\textwidth}
  \centering
  \includegraphics[width=1.0\linewidth]{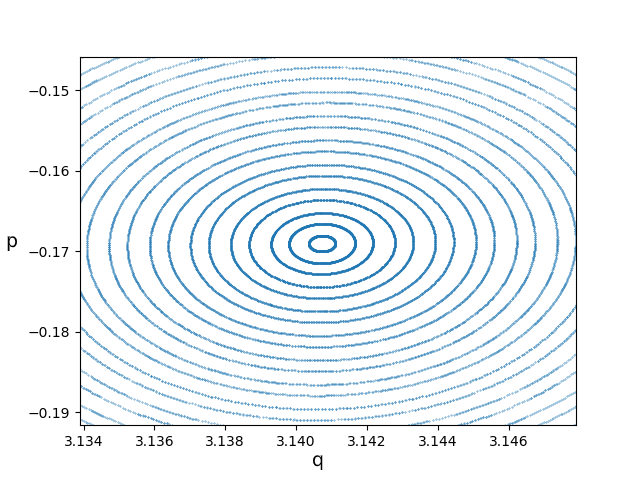}
  \caption{}
  \label{fig:sfig1}
\end{subfigure}\\
\begin{subfigure}{.46\textwidth}
  \centering
  \includegraphics[width=1.0\linewidth]{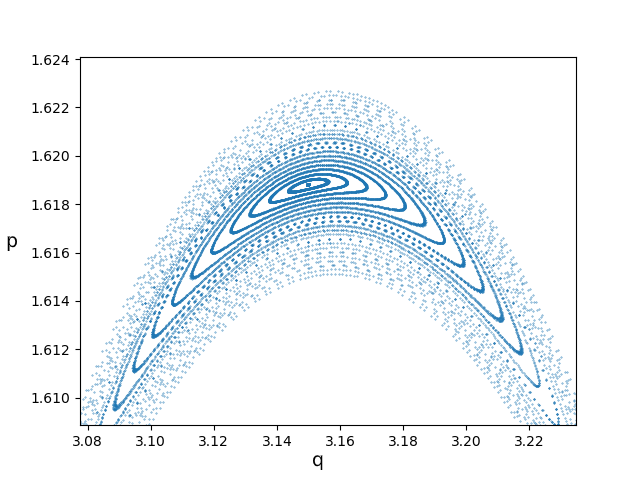}
  \caption{}
  \label{fig:sfig2}
\end{subfigure}%
\begin{subfigure}{.46\textwidth}
  \centering
  \includegraphics[width=1.0\linewidth]{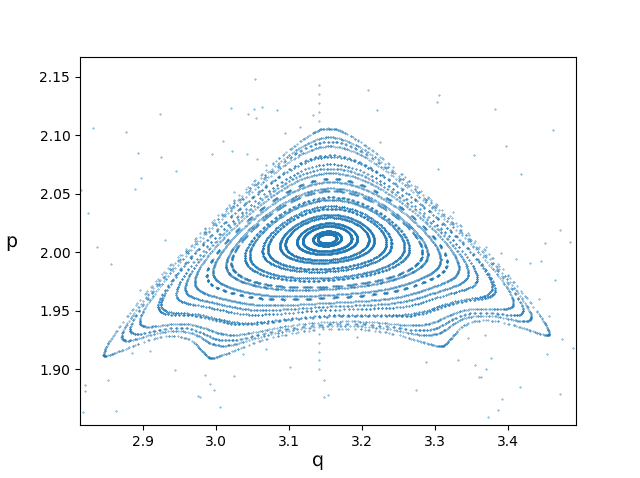}
  \caption{}
  \label{fig:sfig1}
\end{subfigure}
\label{figEarth}
\caption{Poincar{\'e} sections for $\mu = 0$, $k = 10$, $\omega = 4$, $c = 0$, $a = 1$, $A = 1$. Subfigures (b), (c) and (d) represent enlarged regions of (a).}
\end{figure}

\begin{figure}[h!]
\centering
\begin{subfigure}{.46\textwidth}
  \centering
  \includegraphics[width=1.0\linewidth]{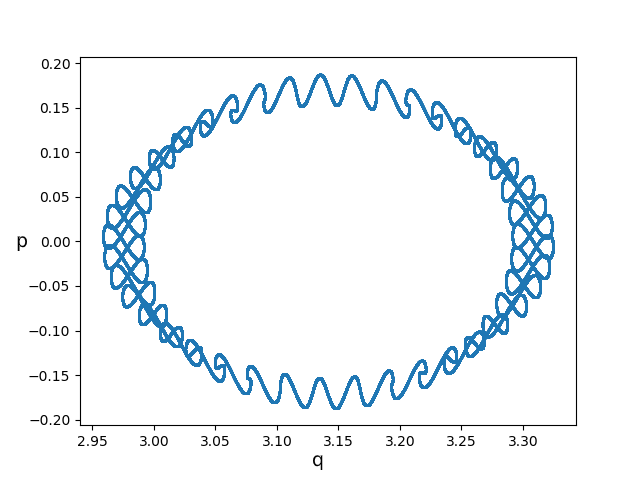}
  \caption{$q(0) \approx 3.14$, $p(0) \approx -0.16$}
  \label{fig:sfig2}
\end{subfigure}%
\begin{subfigure}{.46\textwidth}
  \centering
  \includegraphics[width=1.0\linewidth]{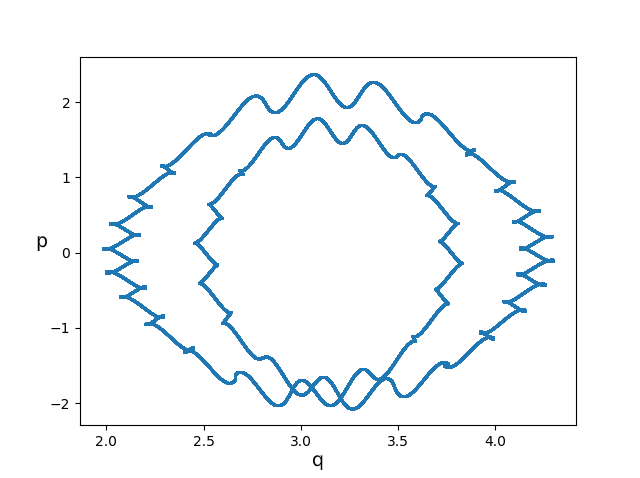}
  \caption{$q(0) \approx 3.15$, $p(0) \approx 1.61$}
  \label{fig:sfig1}
\end{subfigure}
\label{figEarth}
\caption{Periodic solutions without falls for $\mu = 0$, $k = 10$, $\omega = 4$, $c = 0$, $a = 1$, $A = 1$. }
\end{figure}

On Fig. 3 one can find neighborhoods of three $2\pi$-periodic stable solutions. Just three of them are solutions without falls (Fig. 4).

In conclusion, we consider a few more asymptotically stable $2\pi$-periodic solutions that can be obtained as follows. Note that, for any given motion $q(t)$ of the pendulum such that $q(t) \in (\pi/2,3\pi/2)$ for all $t$, it is always possible to choose such an external force $h(t)$ that $q(t)$ is a solution of the corresponding equations:

\begin{align}
\label{eq8}
    h(t) = \frac{1}{\cos q} \left( \dot p + \mu p + [\mu \sin q + p\cos q] \dot f + \sin q + \frac{\dot f^2}{2}\sin 2q \right).
\end{align}

For instance, we will consider motions of the pendulum of the following form

$$
q(t) = A \sin (t).
$$
Let the coefficient of friction equal $1$. Then the system is fully described by the following set of parameters: $a$, $\omega$, $A$, and $1/k$ (small parameter).

It is worth to mention that from \eqref{eq8} it follows that there is a term proportional to $k$ (large parameter) in the expression for $h(t)$. Hence, the classical results on averaging cannot be applied for this system. In particular, Theorems 2.1 and 2.2 also cannot be applied. However, generalizations of these theorems are out of the scope of this paper and will be developed elsewhere. Nevertheless, below we present some numerical examples of asymptotically stable $2\pi$-periodic solutions without falls (Fig. 5).

\begin{figure}[h!]
\begin{subfigure}{.5\textwidth}
  \centering
  \includegraphics[width=1.0\linewidth]{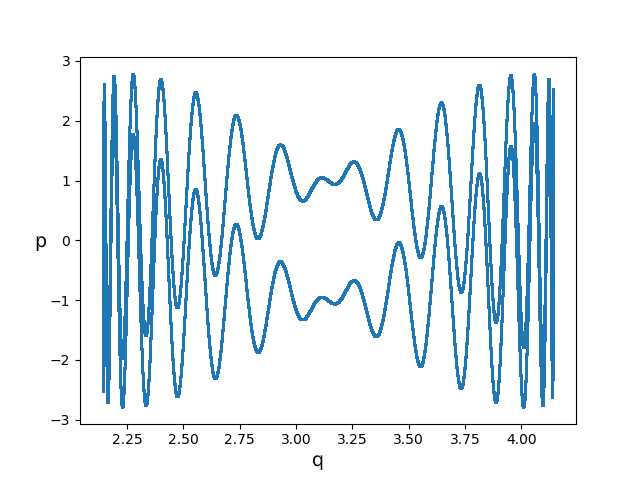}
  \caption{$k = 10$, $\omega = 3$, $a = 1$, $A = 1$.}
  \label{fig:sfig2}
\end{subfigure}%
\begin{subfigure}{.5\textwidth}
  \centering
  \includegraphics[width=1.0\linewidth]{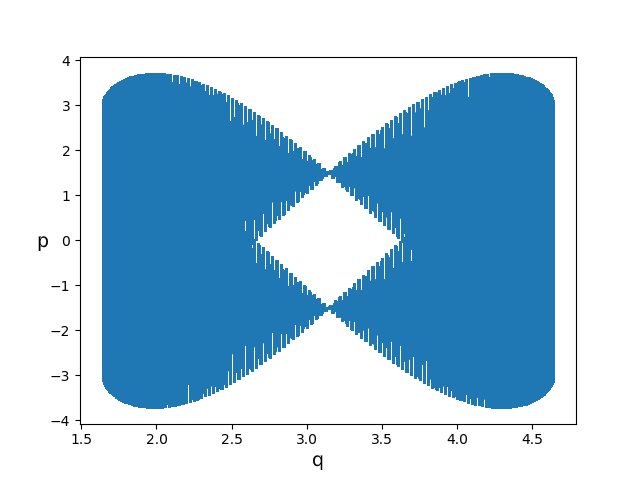}
  \caption{$k = 100$, $\omega = 3$, $a = 1$, $A = 3/2$.}
  \label{fig:sfig1}
\end{subfigure}\\
\begin{subfigure}{.5\textwidth}
  \centering
  \includegraphics[width=1.0\linewidth]{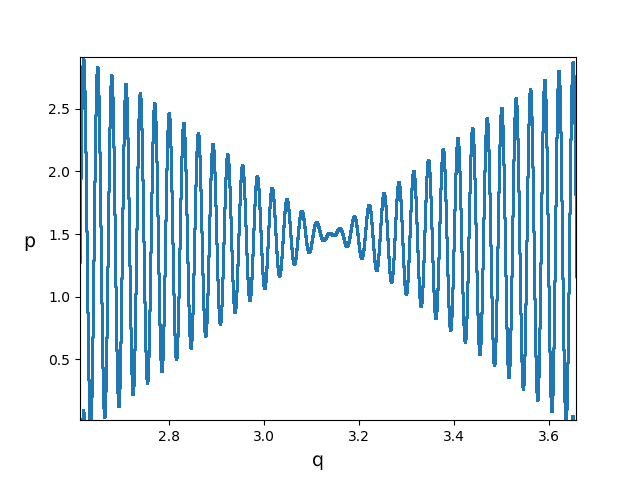}
  \caption{$k = 100$, $\omega = 3$, $a = 1$, $A = 3/2$ (an enlarged region).}
  \label{fig:sfig2}
\end{subfigure}%
\begin{subfigure}{.5\textwidth}
  \centering
  \includegraphics[width=1.0\linewidth]{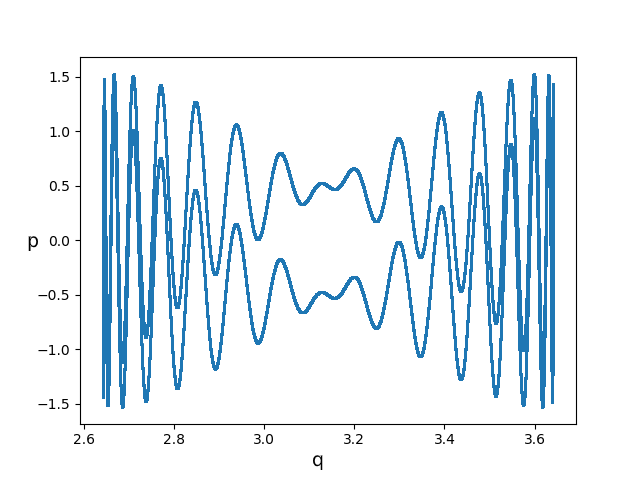}
  \caption{$k = 10$, $\omega = 3$, $a = 1$, $A = 1/2$.}
  \label{fig:sfig1}
\end{subfigure}\\
\label{figEarth}
\caption{Asymptotically stable solutions of the form $A \sin(t)$.}
\end{figure}

\section{Conclusion and remarks}

From the classical results on averaging, it follows that the solutions of the original and the averaged problems, starting at the same point, remain $\varepsilon$-close on some finite time interval, provided $\varepsilon$ is sufficiently small \cite{bogolyubov1961asymptotic}. For an infinite time interval, this statement does not hold. In other words, in a general case, the original and averaged solutions may drift apart significantly for any given positive $\varepsilon$.

In \cite{bogolyubov1961asymptotic} N.N. Bogolyubov and Y. A. Mitropolskij wrote: `\textit{One can, for instance, try to find conditions under which the difference between the exact solution and its asymptotic approximation, for small values of the parameter, becomes arbitrarily small on an arbitrarily long, yet finite, time interval. It is also possible to consider far more difficult problems trying to find a correspondence between such properties of the exact and asymptotic solutions that depends on their behavior on an infinite time interval.}'

The main result on the averaging on an infinite interval is the theorem that states that in a vicinity of a hyperbolic equilibrium of the averaged system there exists a solution of the original system. As it can be seen in the above figures and as it is understood from the nature of the method of averaging, the solution of the original system can be considered as a solution of the averaged system plus some perturbation, that may have a significant impact on a long time interval (see, for instance, \cite{cox2020ponderomotive,yang2020averaging}). 
Under some additional assumptions, the solution of the original system will be periodic. It is important to note that the solutions of the original and averaged system are not assumed to have the same initial conditions.

Here we can see an analogy between this classical result and Theorem 2.1. Let function $h$ in system \eqref{eq4} does not depend on $t$ and the assumptions of Theorem 2.1 hold. Then we obtain that in some (possibly, not small) vicinity of the vertical equilibrium there exists a periodic solution. If we additionally assume that $h(\pi,0)=0$, then we can say that in some vicinity of the vertical equilibrium there is a periodic solution. Moreover, this vicinity can be chosen to be arbitrarily small, provided the norm of $h$ is relatively small.

Note that the conditions, that we impose on the system in order to apply the theorem on the upper and lower solutions, can be considered as an analogue for the hyperbolicity of a solution. The same conditions play the key role when one wants to apply the so-called topological Wa\.zewski method for similar systems \cite{polekhin2015forced,polekhin2016forced,polekhin2018impossibility,polekhin2018topological,srzednicki2019periodic}. In this context, some subset of the extended phase space is called an isolating segment for a given system of equations. An important property of isolating segments is that the existence of a periodic solution depends only on the behavior of the vector field at the boundary of these segments (and does not depend on the vector field inside the segment). This allows us to obtain results on averaging on an infinite time interval based only on the results on local averaging.

The presented method of proof of the existence of periodic solutions in no sense can be regarded as constructive. Therefore, speaking formally, we cannot claim that the stable periodic solutions (that were found numerically) and the solutions, the existence of which is proved in Theorems 2.1 and 2.2, are actually different solutions. However, for the case when $\mu = 0$, the periodic solutions can be found based on the variational approach by means of the gradient descent method. And it can be shown that these solutions are always unstable \cite{bolotin2015calculus}. In particular, every solution obtained from Theorem 2.2 will be unstable, provided there is no external horizontal force. Therefore, the question of whether there always exists a stable periodic solution without falls is a quite important area of future research. Or, to be more precise, the question is to find sufficient conditions for the existence of stable periodic solutions without falls. Note that the so-called nodding solutions mentioned above are firstly, stable and secondly, can be solutions without falls \cite{acheson1995multiple,butikov2001dynamic}.

Another possible area for future research is the study of stability of given periodic solutions without falls. As it was mentioned above, in this case we cannot apply the classical results on averaging. However, there exist several results on averaging for systems with a large parameter (see, for instance, \cite{volosov1961method,levenshtam2005asymptotic,levenshtam2005asymptotic2}) that can be applied to this problems and also to the problem of averaging on an infinite time interval when function $h$ is not periodic in $t$.

In conclusion, we would like to note that the presented approach can be carried over directly to other pendulum-like systems. The main requirement for the application of this method is the possibility to prove the existence of periodic solutions based on the behavior of the vector field in a vicinity of some subset of the extended phase space. In the above considerations we used the method of upper and lower solutions. This method can be replaced with results based on the Wa\.zewski method \cite{wazewski1947principe} and the Lefschetz-Hopf theorem \cite{srzednicki1994periodic,srzednicki2005fixed} or one can use here the dynamical convexity of our system \cite{bolotin2015calculus}.

\section*{Acknowledgments}
The work has been supported by the Grant of the President of the Russian Federation (Project MK-1826.2020.1).

\bibliographystyle{model1-num-names}
\bibliography{sample.bib}

\end{document}